\documentclass[13pts, reqno]{amsart}
\usepackage{amsmath,amscd,amstext,amsthm,amsfonts,latexsym}
\usepackage[dvips]{graphicx}
\usepackage{graphicx, graphics}
\usepackage{color}
\usepackage[active]{srcltx}
\usepackage{a4wide}
\usepackage{amssymb, amsmath}
\usepackage{graphicx}
\usepackage{float}
\usepackage[active]{srcltx}

\usepackage[colorlinks=true, linkcolor=blue, urlcolor=red, citecolor=blue, hyperindex]{hyperref}

\theoremstyle{plain}
\newtheorem{theorem}{Theorem}[section]
\newtheorem{proposition}[theorem]{Proposition}
\newtheorem{lemma}[theorem]{Lemma}
\newtheorem{corollary}[theorem]{Corollary}
\newtheorem{maintheorem}{Theorem}

\theoremstyle{definition}
\newtheorem{remark}[theorem]{Remark}
\newtheorem{example}[theorem]{Example}
\newtheorem{definition}[theorem]{Definition}

\newcommand{\vep}{\delta}

\newcommand{\cF}{{\mathcal F}}

\newcommand{\supp}{\operatorname{{supp}}}

\title[Quantitative recurrence for free semigroup actions]{Quantitative recurrence for free semigroup actions}

\begin{document}

\author[M. Carvalho]{Maria Carvalho}
\address{Centro de Matemática, Universidade do Porto, Portugal.}
\email{mpcarval@fc.up.pt}

\author[F. Rodrigues]{Fagner B. Rodrigues}
\address{Departamento de Matem\'atica, Universidade Federal do Rio Grande do Sul, Brazil.}
\email{fagnerbernardini@gmail.com}

\author[P.Varandas]{Paulo Varandas}
\address{Departamento de Matem\'atica, Universidade Federal da Bahia, Brazil.}
\email{paulo.varandas@ufba.br}



\maketitle
\begin{abstract} We consider finitely generated free semigroup actions on a compact metric space and obtain quantitative information on Poincar\'e recurrence, average first return time and hitting frequency for the random orbits induced by the semigroup action. Besides, we relate the recurrence to balls with the rates of expansion of the semigroup's generators and the topological entropy of the semigroup action. Finally, we establish a partial variational principle and prove an ergodic optimization for this kind of dynamical action.
\end{abstract}

\setcounter{tocdepth}{1}

\section{Introduction}
The research on partially hyperbolic dynamics brought to the stage iterated systems of functions modeling the behavior within the central manifold. This circumstance led to the study of random dynamical systems and a thorough understanding of the dynamical and ergodic properties of these systems has already been achieved \cite{Saussol}. On the other hand, sequential dynamical systems 
have been introduced to model physical phenomena: instead of iterating the same dynamics, one allows the system that describes the real events to readjust with time, and one may work with a family of randomly chosen transformations in a way that matches the inevitable experimental errors \cite{KR11}. However, it is not yet clear how the classical results on first hitting or return times may be generalized to stationary and non-autonomous sequences of maps.

In this work we aim at an extension of the quantitative analysis of Poincar\'e recurrence to the realm of 
finitely generated free semigroup actions. In this context, a first important contribution was obtained in \cite{MR}, where the authors proved that, for rapidly mixing systems, the quenched recurrence rates are equal to the pointwise dimensions of a stationary measure. One should also refer \cite{AFV, RT, RSV} on the distribution of hitting times and extreme laws for random dynamical systems. Equally significant are the recent advances 
obtained in \cite{Abadi, Haydn, Saussol, GRS15}.
Ultimately, we are concerned with the description of the fastest return time when considering all the semigroup elements instead of a single dynamical system. In a recent work \cite{RoVa1}, it has been introduced a notion of topological entropy and pressure for finitely generated continuous free semigroup actions on a compact metric space. Later, in \cite{CRV}, it has been shown that a free semigroup action of either $C^1$ expanding maps, or, more generally, Ruelle-expanding transformations, has a unique measure of maximal entropy which is linked to annealed equilibrium states for random dynamical systems \cite{Baladi}. The main strategy to deal with such a system has been the codification of the random orbits by a true dynamics, namely the skew product based on a full shift with finitely many symbols. Keeping this approach in mind, here we address a few questions regarding recurrence, first return or hitting time maps, and the connection between the rate of frequency of visits to a set, its size, the entropy of the semigroup action and the Lyapunov exponents of the generators.

We will start proving that almost every point is recurrent either by random dynamical systems or by stationary sequential dynamics. Then we will establish a Kac-like property for such return times and estimate an upper bound for the Poincar\'e recurrence to balls, linking the latter to the quenched pressure of random dynamical systems. As return times are strongly related to other dynamically significant quantities, like entropy and Lyapunov exponents, we will also show that, in the case of random dynamical systems generated by expanding maps, the shortest fibred return time to dynamic balls grows linearly, which implies that typical fibred return times to balls may be expressed in terms of the random Lyapunov exponents of the dynamics and, consequently, are independent of the point.  
Moreover, we shall study the connection between the maximum hitting frequency/fastest mean return time to a set with its size when estimated by different invariant measures, extending the ergodic optimization obtained in \cite{Je} to the random context we are considering. Finally, we will introduce the notion of measure-theoretic entropy of a semigroup action and obtain a partial variational principle which improves the estimate in \cite{LMW2016} and complements \cite{Bis, Bis13}. We refer the reader to Subsection~\ref{sec:statements} for the precise statements of the main results.

\section{Main results}\label{sec:setting}
In this section we describe the free semigroup actions we are interested in and state our major conclusions on the quantitative recurrence within this context. The concepts and results we will consider in this work depend on the fixed set of generators $G_1$ but, to improve the general readability of the paper, we will omit this data in the notation.

\subsection{Setting}
Given a compact metric space $(X,d)$, a finite set of continuous maps $g_i:X \to X$, $i \in \mathcal{P}=\{1,2,\ldots,p\}$, and the finitely generated semigroup $(G,\,\circ)$ with the finite set of generators $G_1=\{id, g_1, g_2, \dots, g_p\}$, we write $G=\bigcup_{n\,\in\, \mathbb N_0} G_n$, where $G_0=\{id\}$ and $\underline g \in G_n$ if and only if $\underline g=g_{i_n} \dots g_{i_2} g_{i_1}$, with $g_{i_j} \in G_1$ (for notational simplicity's sake we will use $g_j \, g_i$ instead of the composition $g_j\,\circ\, g_i$). We note that a semigroup may have multiple generating sets. In what follows, we will assume that the generator set $G_1$ is minimal, meaning that no function $g_j$, for $j = 1, \ldots, p$, can be expressed as a composition of the remaining generators. Observe also that each element $\underline g$ of $G_n$ may be seen as a word which originates from the concatenation of $n$ elements in $G_1$. Yet, different concatenations may generate the same element in $G$. Nevertheless, in most of the computations to be done, we shall consider different concatenations instead of the elements in $G$ they create. One way to interpret this statement is to consider the itinerary map $\iota : \mathbb{F}_p  \to  G $ given by
$$\underline i=i_n \dots i_1  \quad \mapsto \quad \underline g_{\underline i} := g_{i_n} \dots g_{i_1}$$
where $\mathbb{F}_p$ is the free semigroup with $p$ generators, and to regard concatenations on $G$ as images by $\iota$ of paths on $\mathbb{F}_p$.

Set $G_1^* = G_1 \setminus \{id\}$ and, for every $n\ge 1$, let $G_n^*$ denote the space of concatenations of $n$ elements in $G_1^*$. To summon each element $\underline{g}$ of $G^*_n$, we will write $|\underline{g}|=n$ instead of $\underline g\in G^*_n$. In $G$, one consider the semigroup operation of concatenation defined as usual: if $\underline{g}=g_{i_n} \dots g_{i_2} g_{i_1}$ and $\underline{h}=h_{i_m} \dots h_{i_2} h_{i_1}$, where $n=|\underline g|$ and $m=|\underline h|$, then
$\underline{g}\,\underline{h}=g_{i_n} \dots g_{i_2} g_{i_1} h_{i_m} \dots h_{i_2} h_{i_1} \in G_{m+n}^*.$
The finitely generated semigroup $G$ induces an \emph{action} in $X$, say
$$
\begin{array}{rccc}
\mathbb{S} : & G \times X & \to & X \\
	& (g,x) & \mapsto & g(x).
\end{array}
$$
We say that $\mathbb{S}$ is a \emph{semigroup action} if, for any $\underline{g},\,\underline{h}  \in G$ and every $x \in X$, we have $\mathbb{S}(\underline{g}\,\underline{h},x)=\mathbb{S}(\underline{g}, \mathbb{S}(\underline{h},x)).$ The action $\mathbb{S}$ is continuous if the map $\underline g : X \to X$ 
is continuous for any $\underline g \in G$. As usual, $x\in X$ is a \emph{fixed point} for $\underline g \in G$ if $\underline g(x)=x$; the set of these fixed points will be denoted by $\text{Fix}(\underline g)$. A point $x \in X$ is said to be a \emph{periodic point with period $n$} by the action $\mathbb{S}$ if there exist $n \in \mathbb{N}$ and $\underline{g} \in G_n^*$ such that $\underline{g}(x) = x$. Write $\text{Per}(G_n)= \bigcup_{|\underline{g}|=n} \,\text{Fix}(\underline{g})$ for the set of all periodic points with period $n$. Accordingly, $\text{Per}(G)=\bigcup_{n\,\ge\, 1} \text{Per}(G_n)$ will stand for the set of periodic points of the whole semigroup action. We observe that, when $G_1^*=\{f\}$, these definitions coincide with the usual ones for the dynamical system $f$.

The action of semigroups of dynamics has a strong connection with skew products which has been scanned in order to obtain properties of semigroup actions by means of fibred and annealed quantities associated to the skew product dynamics (see e.g. \cite{CRV}). We recall that, if $X$ is a compact metric space 
and one considers a finite set of continuous maps $g_i:X \to X$, $i \in \mathcal{P}=\{1,2,\ldots,p\}$, $p\ge 1$, we have defined a skew product dynamics
\begin{equation}\label{de.Skew product}
\begin{array}{rccc}
\mathcal{F}_G : & \Sigma_p^+  \times X & \to & \Sigma_p^+  \times X \\
	& (\omega,x) & \mapsto & (\sigma(\omega), g_{\omega_1}(x))
\end{array}
\end{equation}
where $\omega=(\omega_1,\omega_2, \dots)$ is an element of the full unilateral space of sequences
$\Sigma_p^+ = \mathcal{P}^\mathbb{N}$ and $\sigma$ denotes the shift map on $\Sigma_p^+$. We will write
$\cF_G^n(\omega,x) = (\sigma^n(\omega), f_\omega^n(x))$ for every $n\ge 1$.


In what follows, we will denote by $\mathcal{M}_G$ the set of Borel probability measures on $X$ invariant by $g_i$ for all $i \in \{1,\cdots,p\}$. And $\mathbb{P}_{\underline{a}}$ will stand for the Bernoulli probability measure in $\Sigma^+_p$ which is the Borel product measure determined by a vector $\underline{a}=(a_1,\cdots, a_p)$ satisfying $0<a_i<1$ for every $i \in \{1,2,\cdots,p\}$ and $\sum_{i=1}^p\,a_i=1$.

\subsection{Statements}\label{sec:statements}

As a semigroup action is not a classical dynamical system, but rather an action of several dynamics in the same ambient space which are selected randomly according to some probability measure, the possible notions of recurrence must be carefully chosen and one needs to guarantee that recurrence actually happens. In what follows, we shall examine recurrence either from the point of view of individual concatenations of maps (associated to individual infinite paths in the free semigroup) or by estimating the fastest return times (the smallest return time associated to any of the dynamics in the semigroup).

\subsubsection{\textbf{\emph{Poincar\'e recurrence for sequences of stationary maps}}}

While using infinite concatenations of elements in $G_1$, it is natural to consider the shift space
$\Sigma_p^+ =\{1, \dots p\}^\mathbb N$. Any sequence $\omega\in \Sigma_p^+$ determines
a sequential dynamical system $(g_{\omega_i})_{i\,\in\, \mathbb N}$ and their compositions
\begin{equation*}\label{de.semigroup-orbit}
n\ge 1 \,\mapsto\, f_\omega^n=g_{\omega_n} \dots g_{\omega_2} g_{\omega_1}.
\end{equation*}
For any random walk $\mathbb P$ on $\mathbb F_p$ one expects to find
generic paths for which the dynamics in $X$ exhibits recurrence, meaning that, if one disregards the first shift iterations of the sequence $\omega \in \Sigma_p^+$, then almost
every point in $X$ returns infinitely often by the shifted stationary sequence of maps (a notion that generalizes periodicity). 
Our first result asserts that this is indeed the case.

\begin{maintheorem}\label{Thm:A}
Let $G$ be a finitely generated free semigroup, $\mathbb{S}$ be the corresponding continuous semigroup action, $\nu$ be a Borel probability measure invariant by every generator in $G_1^*$ and $\mathbb{P}$ be a $\sigma$-invariant Borel probability measure on $\Sigma_p^+$. Then, for any measurable subset $A \subset X$ the following properties hold:
\begin{enumerate}
\item For any $\omega \in \Sigma_p^+$, the set of points $x \in A$ for which there are positive integers $n \geq k$ satisfying $g_{\omega_n}g_{\omega_{n-1}}\ldots g_{\omega_k}(x) \in A$ has full $\nu$-measure in $A$.
\item For $\mathbb{P}$-almost every $\omega \in \Sigma_p^+$, the set of the points $x \in A$ whose orbit $\left(f^k_\omega(x)\right)_{k \,\in\, \mathbb{N}}$ returns to $A$ infinitely often has full $\nu$-measure in $A$.
\end{enumerate}
\end{maintheorem}

\subsubsection{\textbf{\emph{Kac expected return time}}}\label{sse:Kac}

Given a measurable map $f: X\to X$ preserving an ergodic probability measure $\nu$, Kac's Lemma asserts that the expected first return time to a positive measure set $A\subset X$ is $\frac{1}{\nu(A)}$.
More precisely, if $\nu(A)>0$ and the first hitting time of $x$ to $A$ is defined by
\begin{equation*}\label{de.1st return}
n_A(x)=\left\{ \begin{array}{ll}
\inf\,\left\{k \in \mathbb{N} \colon f^k(x) \in A\right\} & \mbox{if this set is nonempty}\\
+\infty & \mbox{otherwise}
\end{array}
\right.
\end{equation*}
then $n_A$ is $\nu$-integrable and
\begin{equation}\label{eq:Kac formula}
\int_A \,n_A(x) \, d\nu_A = \frac{1}{\nu(A)}
\end{equation}
where $\nu_A = \frac{\nu}{\nu(A)}$ is the normalized probability in $A$. A version of Kac's Lemma for suspension flows may be found in \cite{Var16}.

In view of Theorem~\ref{Thm:A}(2), it is natural to define, for each measurable $A \subset X$ and $\mathbb{P}$-almost every $\omega \in \Sigma_p^+$, the \emph{first return time to $A$ of $x \in A$} by the dynamics
$\left(f_\omega^k\right)_{k \,\in\, \mathbb{N}_0}$ as follows:
\begin{equation}\label{de.1st ret}
n^\omega_A(x)= \left\{ \begin{array}{ll}
\inf\,\left\{k \in \mathbb{N} \colon f_\omega^k(x) \in A\right\} & \mbox{if this set is nonempty}\\
+\infty & \mbox{otherwise.}
\end{array}
\right.
\end{equation}

We say that the semigroup action $\mathbb S$ is \emph{ergodic with respect to $\mathbb P$ and $\nu$} if the measure $\mathbb P \times \nu$ is ergodic with respect to $\mathcal{F}_G$. This assumption 
is somehow demanding, implying, in particular, that $\mathbb{P}$ is ergodic with respect to $\sigma$. In most instances, however, we will only need to assume that $\mathbb{P}$ is ergodic and that, for any set $A\subset X$ such that $g_i^{-1}(A) = A$ for all $1\le i \le p$, we have $\nu(A) \times \nu(X \setminus A) = 0$.

In Section~\ref{se:Kac} we will show that ergodic semigroup actions 
satisfy a Kac-like recurrence property: the average asymptotic behavior, as $k$ tends to infinity, of the expected first return time to $A$ by the sequence $\big(f^n_{\sigma^k(\omega)}\big)_{n \, \in \, \mathbb N} = \left(g_{\omega_n}\circ \dots \circ g_{\omega_{k+1}}\circ g_{\omega_{k}}\right)_{n \, \in \, \mathbb N}$ is precisely $\frac1{\nu(A)}$. In particular:

\begin{maintheorem}\label{Thm:B}
Let $G$ be a finitely generated free semigroup endowed with a Bernoulli probability measure $\mathbb{P}_{\underline{a}}$ and $\mathbb{S}$ be the corresponding continuous semigroup action. Consider a Borel probability measure $\nu$ in $X$ invariant by every generator in $G$ and assume that $\mathbb{S}$ is ergodic with respect to $\mathbb P_{\underline{a}}$ and $\nu$. Then, for $\mathbb{P}_a$-almost every $\omega$,
$\lim_{k \,\to\, +\infty}\, \int_A \,n_A^{\sigma^k(\omega)}(x) \, d\nu_A(x) = \frac1{\nu(A)}.$
Moreover, there exists a Baire residual subset $\mathcal R \subset \Sigma_p^+$ such that, for every $\omega\in\mathcal R$,
$$\int_A \,n_A^{\omega}(x) \, d\nu_A(x) = \frac1{\nu(A)}.$$
\end{maintheorem}

As a consequence of the last claim, there exists a dense set of values $\omega\in \Sigma_p^+$ for which the non-autonomous dynamics $(f_\omega^n)_{n\ge 1}$ satisfy Kac's formula \eqref{eq:Kac formula}. It is still an open question to determine  
whether this formula holds for $\mathbb P$-almost every $\omega$.

\subsubsection{\textbf{\emph{Partial variational principle}}}

The formula of Abramov and Rokhlin \cite{AbR62} for the measure theoretical entropy of the skew product $\mathcal{F}_G$ with respect to the product measure $\mathbb{P} \times \nu$ suggests a way to define a fibred notion of \emph{metric entropy of a free semigroup action} with respect to a random walk on $\Sigma_p^+$ and an invariant measure on $X$, which we will denote by $h_\nu(\mathbb{S},\mathbb{P})$. This will be done in Subsection~\ref{sse.metric entropy}, just before proving a partial variational principle which extends Theorem 1.2 of \cite{LMW2016} to non-symmetric random walks. Meanwhile, recall that $P_{\text{top}}^{(q)}(\cF_G,0,\mathbb{P})$ stands for the quenched topological pressure of the skew product $\cF_G$ with respect to the random walk $\mathbb{P}$ (see \cite{Baladi}), $h_{\text{top}}(\mathbb{S})$ is the topological entropy of the free semigroup action $\mathbb{S}$ (cf. definition in \cite{RoVa1}) and $h_{\text{top}}(\mathbb{S},\mathbb{P})$ is the relative topological entropy of the free semigroup action with respect to the random walk $\mathbb{P}$ (see \cite{CRV}).

\begin{maintheorem}\label{Thm:C}
Let $\mathbb{S}$ be a finitely generated free semigroup action with generators $G_1=\{id, g_1, \dots, g_p\}$ and consider a Borel $\sigma$-invariant probability measure $\mathbb{P}$ on $\Sigma_p^+$. Then
$$\sup_{\nu \,\in\, \mathcal{M}_G} \, h_\nu(\mathbb{S},\mathbb{P}) \leq h_{\text{top}} (\mathbb{S}) + \left(\log p - h_\mathbb{P}(\sigma)\right).$$
If, additionally, each generator $g_i$ is $C^2$ expanding ($1 \le i \le p$) and $\mathbb{P}=\mathbb{P}_{\underline{a}}$, then
\begin{equation}\label{eq:vp}
\sup_{\nu \,\in\, \mathcal{M}_G} \, h_\nu(\mathbb{S},\mathbb{P}_{\underline{a}}) \leq P_{\text{top}}^{(q)}(\cF_G,0,\mathbb{P}_{\underline{a}}) \leq h_{\text{top}}(\mathbb{S},\mathbb{P}_{\underline{a}}).
\end{equation}
\end{maintheorem}

We remark that the second inequality in \eqref{eq:vp} may be strict, as shown by Example~\ref{example1}.

\subsubsection{\textbf{\emph{Poincar\'e recurrence of balls}}}

In this subsection we will refer to return times of a set to itself by concatenations $f_\omega^n$ of dynamics in $G_1$ associated to a fixed $\omega\in \Sigma_p^+$. Given $A\subset X$ and $\omega\in \Sigma_p^+$, the \emph{$\omega$-shortest return time of $A$ to itself} is defined by
\begin{equation}\label{de:fibre-return}
\mathcal{T}^\omega(A) = \left\{\begin{array}{ll}
\inf\, \{k\in \mathbb{N} \colon f_\omega^k (A) \cap  A \neq\emptyset \} & \mbox{if this set is nonempty}\\
+\infty & \mbox{otherwise.}
\end{array}
\right.
\end{equation}
The shortest return time of the ball $B_\vep(x)$ by the semigroup action $\mathbb{S}$ is equal to
\begin{equation}\label{de:action-return}
\mathcal{T}^{\mathbb{S}}(B_\vep(x)) = \inf\, \{k\in \mathbb{N} \colon \,\exists\,\, \underline g\in G_k^* \colon \,\underline g(B_\vep(x))\cap B_\vep(x) \neq \emptyset\}
\end{equation}
whenever this set is nonempty. Or, equivalently,
$$\mathcal{T}^\mathbb{S}(B_\vep(x)) = \inf_{\omega\, \in\, \Sigma_p^+}\,\mathcal{T}^\omega (B_\vep(x)).$$
Concerning this concept, the next result asserts that, for $\mathbb P$-typical infinite concatenations of dynamics, the minimal returns of dynamical balls grow linearly with the radius, similarly to what happens with a single dynamical system satisfying the orbital specification property and having positive entropy (cf. ~\cite[Theorem~1]{ACS03} and \cite[Theorem B]{Var09} for the case of return times to cylinders and dynamic balls, respectively). 

\begin{maintheorem}\label{Thm:D}
Let $G$ be the semigroup generated by $G_1=\{id, g_1, \dots, g_p\}$, where the elements in $G_1^*$ are $C^1$ expanding maps on a compact connected Riemannian manifold $X$ preserving a common Borel probability measure $\nu$. Consider the continuous semigroup action $\mathbb{S}$ induced by $G$ and a $\sigma$-invariant probability measure $\mathbb P$ on $\Sigma_p^+$. If $h_\nu(\mathbb S,\mathbb P)>0$,  
then, for $\nu$-almost every $x\in X$, one has
\begin{equation}\label{eq:SemigroupUpper}
\limsup_{\delta \,\to\, 0} \,\frac{\mathcal{T}^\mathbb{S}(B_\vep(x)) }{-\log \vep} \le \frac{1}{\log \lambda}
\end{equation}
where $\lambda= \min_{\substack{1\,\le\, i \,\le\, p}}\, \|Dg_i\|.$
If, in addition, all elements in $G^*_1$ are conformal maps, $\mathbb{P}=\mathbb{P}_{\underline a}$ and the semigroup action is ergodic with respect to $\mathbb P_{\underline a}$ and $\nu$, 
then, for $\mathbb P_{\underline a}$-almost every $\omega$ and $\nu$-almost every $x\in X$,
\begin{equation}\label{eq:mc2a}
\lim_{\delta\,\to\, 0} \,\frac{ \mathcal{T}^\omega (B_\delta(x)) }{-\log \vep}	
	= \frac{\dim X}{\sum_{i=1}^p a_i \int \log |\det Dg_i| \, d\nu}
\end{equation}
where $\dim X$ stands for the dimension of the manifold $X$.
\end{maintheorem}

We note that, 
in the special case of finitely generated semigroups of conformal expanding maps for which $|\det Dg_i(\cdot)|$ is constant for every $1\le i \le p$, the expression in the denominator of the right hand-side of \eqref{eq:mc2a} coincides with the quenched pressure of the skew product $\mathcal{F}_G$ with respect to the null observable and the random walk $\mathbb{P}_{\underline a}$ (cf. definition in \cite{Baladi}), and this is bounded above by the topological entropy of the semigroup action with respect to $\mathbb{P}_{\underline{a}}$ (cf. definition in \cite{CRV}). Consequently, in this setting, for $\mathbb P_{\underline a}$-almost every $\omega$ and $\nu$-almost every $x\in X$, we obtain
\begin{equation*}
\lim_{\delta\,\to\, 0} \,\frac{ \mathcal{T}^\omega (B_\delta(x))}{-\log \vep}	
	\ge \frac{\dim X}{h_{\text{top}}(\mathbb{S},\mathbb{P}_{\underline{a}})}>0.
\end{equation*}

\subsubsection{\textbf{\emph{Ergodic optimization}}}

Our last result, inspired by \cite{Je}, deals with the relation between the maximum hitting frequency, the essential
maximal mean return time and the size of a set when measured by different measures. For the required definitions and the proof we refer the reader to Section~\ref{se:RQR}.

\begin{maintheorem}\label{Thm:E}
Let $\mathbb{S}$ be a finitely generated free semigroup action with generators $G_1=\{id, g_1, \dots, g_p\}$ and $\mathbb{P}$ be a Borel $\sigma$-invariant probability measure on $\Sigma^+_p$.
For every closed set $A\subset X$ there exists a marginal $\nu$ on $X$ 
such that
$$\mathbb{P}-\text{esssup}\sup_{x\,\in\, X} \limsup_{n\,\to\,+\infty}\,\frac{\sharp\,\{0\leq i\leq n-1: f_\omega^i(x)\in A\}}{n} = \nu(A).$$
\end{maintheorem}

\section{Proof of Theorem~\ref{Thm:A}}\label{se:Recurrence}

Let $\mathbb{S}: G\times X \to X$ be a semigroup action generated by a finite set $\{g_1,g_2,\ldots,g_p\}$ of $p \geq 2$ dynamics acting on a compact metric space $X$ endowed with a Borel probability measure $\nu$ which is invariant by $g_i$ for every $i\in \{1,2,\ldots,p\}$. Consider the shift map $\sigma$ on the full unilateral space of sequences $\Sigma_p^+ = \{1,2,\ldots,p\}^\mathbb{N}$ and a $\sigma$-invariant Borel probability measure $\mathbb{P}$ in $\Sigma_p^+$. The corresponding skew product $\mathcal{F}_G : \Sigma_p^+ \times X \rightarrow \Sigma_p^+ \times X$ has been defined in \eqref{de.Skew product} and preserves the probability measure $\mathbb{P} \times \nu$.

\subsection{Random Ergodic Theorem}\label{thm:RET}

Let us recall a generalized ergodic theorem from \cite{Halmos}. Let $Z$ and $X$ be measure spaces with probability measures $\mathbb{P}$ and $\nu$, respectively. Suppose that $U: Z \to Z$ is an $\mathbb{P}$ preserving transformation and  denote by $Q: Z\times X \to Z\times X$ the skew product defined by $Q(z,x)=(U(z),T_z(x))$, where the family $\left(T_z\right)_{z \,\in\, Z}$ is assumed to be measurable and,
for each $z \in Z$, $T_z:X \to X$ is a $\nu$-measure preserving map. The skew product $Q$ is measurable and preserves the probability measure $\mathbb{P} \times \nu$. Write $T^0_z=T_z$ and, for $k \in \mathbb{N}$, $T^k_z= T_{U^k(z)}\,\ldots\,T_{U(z)}\,T_z$. It is not hard to show, using Birkhoff's Ergodic Theorem for the skew product $Q$ and $\mathbb P \times\nu$, that, if $\varphi$ is a $\nu$-integrable function in $X$, then there exists a $(\mathbb P \times \nu)$-full measure subset $E \subset Z \times X$ such that, for every $(z,x) \in E$, the averages
$$\Bigg(\frac{1}{n}\,\sum_{j=0}^{n-1}\,\varphi(T^j_{z}(x))\Bigg)_{n \,\in\, \mathbb{N}}$$
converge to a $\nu$-integrable function $\varphi_z^*$ so that
\begin{eqnarray*}
\int \varphi(x)\,d\nu(x) &=& \int \int  \psi(z,x)\,d\mathbb{P} \, d\nu = \int \int  \psi^*(z,x)\,d\mathbb{P} \, d\nu =  \int \int  \varphi_z^*(x)\,d\mathbb{P}(z) \, d\nu(x).
\end{eqnarray*}
Then, Fubini theorem ensures that, for $\mathbb{P}$-almost every $z \in Z$, the set $E^z$ of points $x \in X$ whose averages $\left(\frac{1}{n}\,\sum_{j=0}^{n-1}\,\varphi(T^j_{z}(x))\right)_{n \,\in\, \mathbb{N}}$ converge to $\varphi_z^*$ has full $\nu$ measure. If, moreover, $\mathbb{P} \times \nu$ is ergodic with respect to the skew product $Q$, then $\varphi_z^*(x) =\int \varphi d\nu$ for $\mathbb{P}$-almost every $z \in Z$ and $\nu$-almost every $x \in X$.

\subsection{Recurrence \emph{via} the skew product}\label{se:recurrence_1}

To get a version of Poincar\'e's Recurrence Theorem for a semigroup action, we will start deducing recurrent properties of stationary non-autonomous sequences of dynamical systems and fibred maps.

\begin{proposition}\label{prop:recskew}
Consider the skew product $\cF_G$, a $\sigma$-invariant probability measure $\mathbb{P}$ on $\Sigma_p^+$ 
and a Borel probability measure $\nu$ in $X$ invariant by every generator in $G$. For any measurable subset $A \subset X$ the following properties hold:
\begin{enumerate}
\item For $\mathbb{P}$-almost every $\omega \in \Sigma_p^+$, the set of the points $x \in A$ whose orbit
	$\left(f^k_\omega(x)\right)_{k \,\in\, \mathbb{N}}$ returns to $A$ infinitely often has full $\nu$-measure in $A$.	
\item For every $\omega \in \Sigma_p^+$, the set of points $x \in A$ for which there are positive integers $n \geq k$ satisfying $g_{\omega_n}g_{\omega_{n-1}}\ldots g_{\omega_k}(x) \in A$ has full $\nu$-measure in $A$.
\item If $\nu$ is ergodic with respect to one of the generators, say $g_1$, then there exists a subset $\Omega \subset \Sigma_p^+$ with $\mathbb{P}(\Omega)>0$ such that for every $\omega \in \Omega$ there is a set $Y_\omega \subset X$ with $\nu(Y_\omega)=1$ so that, for any $x \in Y_\omega$, we may find $\ell=\ell(\omega,x) \in \mathbb{N}$ such that the orbit $\left(f^k_\omega(g_1^\ell(x))\right)_{k \,\in\, \mathbb{N}}$ of $g_1^\ell(x)$ enters infinitely many times in $A$.
\item If $\mathbb{P} \times \nu$ is ergodic with respect to $\mathcal{F}_G$, then for $\mathbb{P}$-almost every $\omega \in \Sigma_p^+$ the orbit $\left(f^k_\omega(x)\right)_{k \,\in\, \mathbb{N}}$ of $\nu$-almost every $x \in X$ enters infinitely many times in $A$.
\end{enumerate}
\end{proposition}

Some comments are in order. Items (1) and (4) provide expected results on the recurrence of almost every point with respect to almost every random path. 
Item (2) indicates that, for any stationary sequence of maps, recurrence surely happens up to a convenient shifting of the orbits. Item (3) imparts a dual statement by replacing this shifting by a finite transient of some generator $g_1$, which is assumed to be ergodic with respect to $\nu$. We also remark that, in the case of finitely generated free abelian semigroups, the generators commute and probability measures invariant by any generator do exist.

\begin{proof}
Given $k \in \mathbb{N}$ and $\omega=\omega_1 \omega_2 \ldots \in \Sigma_p^+$, recall that we write $f^k_\omega= g_{\omega_k} \,g_{\omega_{k-1}} \ldots g_{\omega_1}.$ Let $A$ be a measurable subset of $X$ with $\nu(A)>0$ and consider $\Sigma_p^+ \times A$. As the probability measure $\mathbb{P} \times \nu$ is invariant by the skew product $\cF_G$ and $(\mathbb{P} \times \nu)(\Sigma_p^+ \times A)=\nu(A)>0$, by Poincar\'e's Recurrence Theorem there is a subset $E \subset \Sigma_p^+ \times A$ with $(\mathbb{P} \times \nu)(E)=\nu(A)>0$ such that every $(\omega,x) \in E$ returns to $\Sigma_p^+ \times A$ infinitely often by the iteration of $\mathcal{F}_G$. Observe now that
\begin{equation}\label{eq:recskew}
\mathcal{F}_G^k(\omega,x) = (\sigma^k(\omega), g^k_\omega(x)) \in \Sigma_p^+ \times A
	\quad \Leftrightarrow \quad f^k_\omega(x) \in A
\end{equation}
so the property describing the set $E$ informs that, for every $(\omega,x) \in E$, there are infinitely many values of $k\ge 1$ such that $f^k_\omega(x) \in A$. Besides, by Fubini-Tonelli's Theorem
we have
$$\nu(A)=(\mathbb{P} \times \nu)(E)= \int_X\, \mathbb{P}(E^x) \, d\nu(x) =  \int_{\Sigma_p^+}\, \nu(E^{\omega}) \, d\mathbb{P}(\omega),$$
where $E^x=\{\omega \in \Sigma_p^+ : (\omega, x) \in E\}$ and $E^\omega=\{x \in A : (\omega, x) \in E\}$. Thus, for
$\mathbb{P}$-almost every $\omega \in \Sigma_p^+$, we must have $\nu(E^{\omega})=\nu(A)$. This completes the proof of item (1).

To prove item (2), we will pursue another argument without summoning up the skew product, aiming the recurrence by the (possibly non-generic) random orbits $\left(f^k_\omega\right)_{k \,\in\, \mathbb{N}}$. Given $\omega=\omega_1\omega_2\omega_3 \ldots \in \Sigma_p^+$, write
$$B_\omega=\bigcap_{k \,\ge\, 1} \bigcap_{n \,\geq\, k} \,
	\{x \in A: g_{\omega_n}g_{\omega_{n-1}}\ldots g_{\omega_k}(x) \notin A\}.$$
Points in $B_\omega$ are in $A$ but never return to $A$ by any concatenation of dynamics given by the sequences $(g_{\omega_j})_{j\,\ge\, k}$, for all $k\ge 1$. We claim that $\{(g_{\omega_j} \dots g_{\omega_1})(B_\omega)\}_{j\,\ge\, 1}$
defines a family of pairwise disjoint subsets of $X$. Indeed, given positive integers $m > n$, if $(g_{\omega_m}g_{\omega_{m-1}}\ldots g_{\omega_1})^{-1}(B_\omega) \cap (g_{\omega_n}g_{\omega_{n-1}}\ldots g_{\omega_1})^{-1}(B_\omega) \neq \emptyset$, then there would exist $x \in X$ such that $z=g_{\omega_n}g_{\omega_{n-1}}\ldots g_{\omega_1}(x) \in B_\omega$ as well as $g_{\omega_m}g_{\omega_{m-1}}\ldots g_{\omega_{n+1}}(z) \in B_\omega$, which contradicts the definition of $B_\omega$. As $\nu$ is invariant by $g_i$ for every $i \in \mathcal{P}$, it is also invariant by $g_{\omega_n}g_{\omega_{n-1}}\ldots g_{\omega_1}$ for every $n \in \mathbb{N}$. Therefore,
$$\sum_{n=1}^\infty \,\nu(B_\omega) = \sum_{n=1}^\infty \,\nu\left((g_{\omega_n}g_{\omega_{n-1}}\ldots g_{\omega_1})^{-1}(B_\omega)\right) = \nu\left(\bigcup_{n=1}^\infty \, (g_{\omega_n}g_{\omega_{n-1}}\ldots g_{\omega_1})^{-1}(B_\omega)\right) \leq 1$$
and so $\nu(B_\omega)=0$. Thus, for $\nu$-almost every $x\in A$, there exists $n\ge k \ge 1$ such that $g_{\omega_n}g_{\omega_{n-1}}\ldots g_{\omega_k}(x) \in A$. It is not hard to adapt the previous argument to show that, for every $\omega \in \Sigma_p^+$, there exists a full $\nu$-measure subset of points $x \in A$ which exhibit infinitely many returns to $A$ (that is, which admit infinitely many values $n_\ell \ge k_\ell$ such that $g_{\omega_{n_\ell}}g_{\omega_{n_\ell-1}}\ldots g_{\omega_{k_\ell}}(x) \in A$).

We now focus on item (3). As the probability measure $\mathbb{P} \times \nu$ is invariant by the skew product $\mathcal{F}_G$, we may apply to $\varphi=\chi_A$ the Random Ergodic Theorem quoted in Subsection~\ref{thm:RET}. This way, we conclude that, for $\mathbb{P}$-almost every $\omega \in \Sigma_p^+$, the frequency of visits to $A$ given by
$$
\frac{1}{n}\,\sharp\left\{0\leq j\leq n-1: f^j_{\omega}(x) \in A\right\}
	= \frac{1}{n}\,\sum_{j=0}^{n-1}\,\chi_A(f^j_{\omega}(x))
$$
is convergent for $\nu$-almost every $x \in X$. As $\nu(A)>0$, we may add that, for $\mathbb{P}$-almost every $\omega \in \Sigma_p^+$ and $\nu$-almost every $x \in A$, those averages converge to the value at $x$ of a $\nu$-integrable function $\varphi_\omega^*$ that satisfies
$\int \int \varphi_\omega^*(x)\,d\nu(x)\,d\mathbb{P}(\omega)= \int \varphi(x)\,d\nu(x)= \nu(A) >0.$
Consequently, the set $C \subset \Sigma_p^+ \times X$ of the points $(\omega,x)$ for which we have $\varphi_\omega^*(x)>0$ satisfies $(\mathbb{P} \times \nu)(C)>0$. Therefore, for every $(\omega,x) \in C$, the point $f^k_{\omega}(x)$ is in $A$ for infinitely many choices of $k \in \mathbb{N}$. By Fubini-Tonelli's Theorem, we get
$(\mathbb{P} \times \nu)(C) = \int_{X}\, \nu(C^\omega) \, d\mathbb{P}(\omega) >0$
where
$C^\omega=\{x \in X : (\omega, x) \in C\} = \{x \in X : f^k_{\omega}(x) \in A \,\,\text{for infinitely many }\, k \in \mathbb{N}\}.$
Thus, there must exist a subset $\Omega \subset  \Sigma_p^+$ with $\mathbb{P}(\Omega)>0$ such that, for every $\omega \in \Omega$, we have $\nu(C^\omega)>0$. Notice, however, that, if $\nu(A)<1$, the previous property is not enough for us to be sure whether $\nu(A\cap C^\omega) > 0$ for some relevant subset of elements in $\Omega$. Nevertheless, under the assumption that $\nu$ is ergodic by one of the generators, say $g_1$, we may take for each $\omega \in \Omega$ the set
$$Y_\omega= \bigcup_{k \,\in\, \mathbb{N}}\, g_1^{-k}(C^\omega)$$
and conclude that, as $Y_\omega \subset g_1^{-1}(Y_\omega)$, we have $\nu(Y_\omega)=1$. That is, for each $\omega \in \Omega$ and every $x \in Y_\omega$, there is $k \in \mathbb{N}$ such that $g_1^k(x) \in C^\omega$.

Concerning item (4), observe that, if $\mathbb{P} \times \nu$ is ergodic with respect to $\mathcal{F}_G$, then $\varphi_{\omega}^*=\int \varphi d\nu=\nu(A)>0$ for $\mathbb{P}$-almost every $\omega \in \Sigma_p^+$  and $\nu$-almost every $x \in X$. That is, $(\mathbb{P} \times \nu)(C)=1$ and there exists a subset $\Omega \subset  \Sigma_p^+$ with $\mathbb{P}(\Omega)=1$ such that, for every $\omega \in \Omega$, we have $\nu(C^\omega)=1$. So, without assuming the ergodicity of $\nu$ with respect to one of the generators, the first part of the argument in the previous paragraph shows that, for $\mathbb{P}$-almost every $\omega \in \Sigma_p^+$ and $\nu$-almost every $x \in X$, the orbit $\left(f^m_\omega(x)\right)_{m \,\in\, \mathbb{N}}$ of $x$ returns infinitely many times to $A$. This completes the proofs of Proposition~\ref{prop:recskew} and Theorem~\ref{Thm:A}.
\end{proof}

\begin{remark}\label{rmk:countable}
The full measure subset mentioned in Proposition~\ref{prop:recskew} depends on the sequential dynamical system ${((f_\omega^n)_{\omega \,\in\, \Sigma^+_p})}_{n\,\ge\, 1}$. Nevertheless, the argument used in its proof contains a stronger statement if the semigroup $G$ is finite or countable (as, for instance $\mathbb Z^p_+$): \emph{if $A$ is a positive $\nu$-measure subset of $X$, then there exists $B\subset A$ such that, for every $\omega \in \Sigma^+_p$, there are positive integers $n \geq k$ such that $g_{\omega_n}g_{\omega_{n-1}}\ldots g_{\omega_k}(x) \in A$.}
\end{remark}

\begin{example}
Let $X$ be a compact connected Riemannian manifold, $m$ stand for the volume measure in $X$, $A\subset X$ be an open set, $\text{Diff}^{\,1}_m(X)$ denote the group of $C^1$ volume preserving diffeomorphisms on $X$ and $G_1\subset \text{Diff}^{\,1}_m(X)$ be a finite set. Then, for any sequence $(f_n)_{n \,\in\, \mathbb{N}}$ in $G_1^{\mathbb N}$, there exists an $m$-full measure subset of points $x\in A$ for which we may find infinitely many positive integers $k_i(x) < \ell_i(x)$ such that $f_{\ell_i} \circ \cdots \circ f_{k_i}(x) \in A$.
\end{example}

\section{Proof of Theorem~\ref{Thm:B}}\label{se:Kac}

Throughout this section we will study recurrence properties for semigroup actions using ergodic information about the skew product $\cF_G$ and the measure $\mathbb{P} \times \nu$.  
Take $\nu \in \mathcal{M}_G$ and 
a Borel $\sigma$-invariant probability measure $\mathbb{P}$. The corresponding skew product $\mathcal{F}_G : \Sigma_p^+ \times X \rightarrow \Sigma_p^+ \times X$ has been defined in \eqref{de.Skew product} and preserves the probability measure $\mathbb{P} \times \nu$. The next result is a quenched version of the expected first return time and provides an averaged fibred Kac's Lemma, from which Theorem~\ref{Thm:B} is a direct consequence.

\begin{proposition}\label{prop:Kacskew} Assume that $\mathbb{P} \times \nu$ is ergodic with respect to the skew product $\cF_G$. Then, given a measurable set $A \subset X$ with $\nu(A)>0$, for $\mathbb{P}$-almost every $\omega$ in $\Sigma_p^+$ one has
$$\lim_{k \,\to\, + \infty}\,\frac{1}{k}\,\sum_{j=0}^{k-1}\,\int_A \,n_A^{\sigma^j(\omega)}(x) \, d\nu_A(x) = \frac1{\nu(A)}.$$
If, in addition, $\mathbb{P}$ is mixing, then, for $\mathbb{P}$-almost every $\omega$,
$$\lim_{k \,\to\, +\infty}\,\int_A \,n_A^{\sigma^k(\omega)}(x) \, d\nu_A(x) = \frac1{\nu(A)}$$
\end{proposition}

\begin{proof}
Firstly, the Proposition~\ref{prop:recskew} ensures that, for $\mathbb{P}$-almost every $\omega \in \Sigma_p^+$, the set $A_\omega$ of points $x \in A$ whose orbit $\left(f^k_\omega(x)\right)_{k \,\in\, \mathbb{N}}$ returns to $A$ infinitely often has full $\nu$-measure in $A$. Therefore, we may consider the map $\varphi : \Sigma_p^+ \to\mathbb R$ defined by
$$\omega \in \Sigma_p^+ \mapsto \varphi(\omega)=\int_A \,n_A^\omega(x) \, d\nu(x)$$
where $n_A^\omega(\cdot)$ denotes the first hitting time to the set $A$ by the sequence $(f_\omega^n)_{n\,\ge\, 1}$ (cf. definition in \eqref{de.1st ret}). The map $\varphi$ is measurable and, as we are assuming that $\mathbb{P} \times \nu$ is ergodic with respect to $\mathcal{F}_G$, then, by Kac's Lemma, $\varphi$ belongs to $L^1(\mathbb{P})$ and
\begin{equation}\label{eq:int}
\int \varphi \, d\mathbb{P} = \int_{\Sigma^+_p} \int_A \,n_A^\omega(x) \, d\nu(x) \,d\mathbb{P}(\omega) =  1.
\end{equation}
Besides, as $\mathbb{P}$ is ergodic (a consequence of the ergodicity of $\mathbb{P} \times \nu$), the application of Birkhoff's Ergodic Theorem to $\varphi$ and $\mathbb{P}$ yields that, for $\mathbb{P}$-almost every $\omega$,
\begin{equation*}\label{eq:int2}
\lim_{k \,\to\, +\infty}\,\frac1k \,\sum_{j=0}^{k-1} \,\int_A \,n_A^{\sigma^j(\omega)}(x) \, d\nu(x) = \int \int_A \,n_A^\omega(x) \, d\nu(x) d\mathbb{P}(\omega) =  1.
\end{equation*}
If $\mathbb{P}$ is mixing, then, for $\mathbb{P}$-almost every $\omega$,
\begin{equation*}\label{eq:mixing}
\lim_{k \,\to\, +\infty}\,\int_A \,n_A^{\sigma^k(\omega)} d\nu(x)= 1.
\end{equation*}
\end{proof}

\begin{lemma}\label{le:semi.c.inf}
The map $\varphi:\Sigma_p^+\to \mathbb R$ is lower semi-continuous.
\end{lemma}

\begin{proof}
First of all, we notice that $n_A^\omega(x) < \infty$ for $\mathbb{P}$-almost every $\omega \in \Sigma^+_p$ and $\nu$-almost every $x\in X$ (cf. Theorem~\ref{Thm:A}). 
For such an $\omega = \omega_1\,\omega_2\,\cdots \in \Sigma_p^+$ and $k \in \mathbb{N}$, let $A_k$ be the set $\{x\in A \colon n_A^\omega(x)=k\}$ and $[\omega_1 \dots \omega_k]$ denote the set of sequences $\theta \in \Sigma_p^+$ such that $\theta_i=\omega_i$ for all $1\leq i\leq k$. Observe also that, if $\theta \in [\omega_1\dots\omega_k]$, then $n_A^\omega(x)=n_A^\theta(x)$ for any $x \in A_k$. Besides, as
$$\int_A \,n_A^\omega(x) \, d\nu(x) = \sum_{k=1}^\infty \,k\,\nu(A_k) <\infty$$
for any $\varepsilon>0$ there exists $N(\varepsilon)\in\mathbb N$ such that $\sum_{k=N(\varepsilon)+1}^{\infty} \,k\,\nu(A_k)<\varepsilon$. Therefore,
$$\varphi(\omega) = \sum_{k=1}^{N(\varepsilon)} \,k\,\nu(A_k) + \sum_{k=N(\varepsilon)+1}^{\infty} \,k\,\nu(A_k) < \sum_{k=1}^{N(\varepsilon)} \,k\,\nu(A_k) + \varepsilon$$
and so
$$ \varphi(\omega)-\varepsilon <\sum_{k=1}^{N(\varepsilon)} \,k\,\nu(A_k)< \sum_{k=1}^{N(\varepsilon)} \,k\,\nu(A_k) + \int_{A \setminus \cup_{k=1}^{N(\varepsilon)}\,A_k}\,n_A^\theta(x) \, d\nu(x)=\int_A \,n_A^\theta(x) \, d\nu(x)=\varphi(\theta).$$
\end{proof}

As $\varphi$ is lower semi-continuous, $\varphi$ has a residual set $\mathcal C$ of points of continuity and there exists $\bar{\omega}$ in the support of the measure $\mathbb P$ where $\varphi$ attains its minimum, that is,
$$\int_A \,n_A^{\bar{\omega}}(x) \, d\nu(x) =\min_{\omega \,\in\, \Sigma^+_p }\,\int_A \,n_A^{\omega}(x)\, d\nu(x).$$
Moreover, if $\mathbb P=\mathbb P_{\underline{a}}$, then it is positive on nonempty open sets and so we may take $\omega_0 \in \mathcal C \cap \supp \mathbb P$. As $\mathbb P_{\underline{a}}$ is
mixing, there exists $\omega \in \Sigma^+_p$ and a sequence $n_k \to \infty$ such that
$\lim_{k\to\infty} d\,(\sigma^{n_k}(\omega), \omega_0) = 0$ and  $\lim_{k \,\to\, +\infty}\,\int_A \,n_A^{\sigma^{n_k}(\omega)} d\nu(x) = 1$. Consequently, as $\omega_0$ is a continuity point of $\varphi$,
$$\int_A \,n_A^{\omega_0} (x) \, d\nu(x) = \varphi(\omega_0) = \lim_{k \,\to\, +\infty}\, \varphi (\sigma^{n_k}(\omega)) = \lim_{k \,\to\, +\infty}\,\int_A \,n_A^{\sigma^{n_k}(\omega)} d\nu(x) = 1.$$

\section{Proof of Theorem~\ref{Thm:C}}\label{se:vp}

Firstly we will introduce the definition of measure-theoretic entropy for a free semigroup action. Afterwards, we will deduce a partial variational principle.

\subsection{Measure-theoretic entropy of a free semigroup action}\label{sse.metric entropy}

Let $\mathbb{P}$ be a $\sigma$-invariant probability measure and $\nu$ a probability measure invariant by any generator in $G_1^*$. Given a measurable finite partition $\beta$ of $X$, $n \in \mathbb{N}$ and $\omega=\omega_1 \omega_2 \cdots \in \Sigma^+_p$, define
\begin{eqnarray}\label{eq.beta}
\beta^n_1(\omega) &=& g_{\omega_1}^{-1} \beta\, \bigvee\, g_{\omega_1}^{-1}g_{\omega_2}^{-1} \beta\, \bigvee\,\cdots\,\bigvee\,  g_{\omega_1}^{-1} g_{\omega_2}^{-1}\cdots g_{\omega_{n-1}}^{-1}\beta \\
\beta^n_0(\omega) &=& \beta\, \bigvee\, \beta^n_1(\omega) \quad \text{and}\quad
\beta^\infty_1(\omega) = \bigvee_{n=1}^\infty\, \beta_1^n(\omega) \nonumber.
\end{eqnarray}
Then the conditional entropy of $\beta$ relative to $\beta^\infty_1(\omega)$, denoted by $H_\nu(\beta|\beta^\infty_1(\omega))$, is a measurable function of $\omega$ and $\mathbb{P}$-integrable (cf. \cite{Petersen}). Let
$h_\nu(\mathbb{S}, \mathbb{P}, \beta)=\int_{\Sigma^+_p}\,H_\nu(\beta|\beta^\infty_1(\omega))\,d\,\mathbb{P}(\omega).$
Proposition 1.1 of \cite[\S6]{Petersen} shows that
\begin{equation}\label{eq:conditional entropy}
h_\nu(\mathbb{S}, \mathbb{P}, \beta)= \lim_{n \,\to\, + \infty}\, \frac{1}{n}\, \int_{\Sigma^+_p}\,H_\nu(\beta^n_0(\omega))\,d\,\mathbb{P}(\omega).
\end{equation}
where $H_\nu(\beta^n_0(\omega))$ is the entropy of the partition $\beta^n_0(\omega)$.

\begin{definition}\label{de.metric entropy} The \emph{metric entropy of the semigroup action with respect to $\mathbb{P}$ and $\nu$} is given by
$$h_\nu(\mathbb{S},\mathbb{P})=\sup_\beta \, h_\nu(\mathbb{S}, \mathbb{P}, \beta).$$
\end{definition}

For instance, if $\mathbb{P}$ is a Dirac measure $\delta_{\underline{j}}$ supported on a fixed point $\underline{j}=jj\cdots$, where $j \in \{1,\cdots,p\}$, then
$h_\nu(\mathbb{S},\delta_{\underline{j}})=h_\nu(g_j).$
If, instead, $\mathbb{P}$ is the symmetric random walk, that is, the Bernoulli $(\frac{1}{p}, \cdots,\frac{1}{p})$-product probability measure $\mathbb{P}_{\underline{p}}$, then (compare with \cite[Definition 4.1]{LMW2016})
$$h_\nu(\mathbb{S},\mathbb{P}_{\underline{p}}) = \sup_\beta\,\lim_{n\, \to\, + \infty}\, \frac{1}{n}\,\left(\frac{1}{p^n}\,\sum_{|\omega| = n} \,H_\nu(\beta^n_0(\omega))\right).$$

Let us resume the proof of Theorem~\ref{Thm:C}. For every $\nu$ and $\mathbb{P}$ as prescribed before, Abramov and Rokhlin proved that
\begin{equation}\label{Abramov-Rokhlin formula}
h_{\mathbb{P} \times \nu}(\cF_G)= h_\mathbb{P}(\sigma) + h_\nu(\mathbb{S},\mathbb{P}).
\end{equation}
If we now summon Bufetov's formula $h_{\text{top}}(\cF_G) = \log p + h_{\text{top}}(\mathbb{S})$ from \cite{Bufetov}
then we conclude that, for every $\sigma$-invariant probability measure $\mathbb{P}$, we have
\begin{eqnarray*}
\sup_{\nu \,\in\, \mathcal{M}_G} \, h_\nu(\mathbb{S},\mathbb{P})
&\leq& \sup_{\nu \,\in\, \mathcal{M}_G} \, \{\,h_{\mathbb{P} \times \nu}(\cF_G) - h_\mathbb{P}(\sigma)\,\}
	= \sup_{\nu \,\in\, \mathcal{M}_G} \, \{\,h_{\mathbb{P} \times \nu}(\cF_G)\,\} - h_\mathbb{P}(\sigma) \\
&\leq& h_{\text{top}}(\cF_G) - h_\mathbb{P}(\sigma)
	= h_{\text{top}} (\mathbb{S}) + \log p - h_\mathbb{P}(\sigma).
\end{eqnarray*}
When $\mathbb{P}=\mathbb{P}_{\underline{p}}$, as $h_{\mathbb{P}_{\underline{p}}}(\sigma)=\log p$, we obtain
$$\sup_{\nu \in \mathcal{M}_G} \, h_\nu(\mathbb{S},\mathbb{P}_{\underline{p}}) \leq h_{\text{top}} (\mathbb{S}).$$
%
%

If each generator $g_i$, for $i = 1,\cdots,p$ is $C^2$ expanding and $\mathbb{P}$ is a Bernoulli probability measure $\mathbb{P}_{\underline{a}}$ for some probability vector $\underline{a}=(a_1,\cdots, a_p)$, then
\begin{eqnarray*}
\sup_{\nu\, \in\, \mathcal{M}_G} \, h_\nu(\mathbb{S},\mathbb{P}_{\underline{a}}) &=& \sup_{\nu\, \in\, \mathcal{M}_G} \, \{\,h_{\mathbb{P}_{\underline{a}} \times \nu}(\cF_G)\,\} - h_{\mathbb{P}_{\underline{a}}}(\sigma) \leq \sup_{\mu:\,\, (\cF_G)_*\mu=\mu,\,\, \pi_*\mu=\mathbb{P}_{\underline{a}}} \, \{\,h_\mu(\cF_G)\,\} - h_{\mathbb{P}_{\underline{a}}}(\sigma) \\
&=& P_{\text{top}}^{(q)}(\cF_G,0,\mathbb{P}_{\underline{a}}) \leq P_{\text{top}}^{(a)}(\cF_G,0,\mathbb{P}_{\underline{a}}) = h_{\text{top}}(\mathbb{S},\mathbb{P}_{\underline{a}}).
\end{eqnarray*}

\begin{remark}
Observe that, when $\underline{a}=\underline{p}$, we have (cf. \cite{CRV})
$$h_{\text{top}}(\mathbb{S},\mathbb{P}_{\underline{p}})= h_{\text{top}}(\mathbb{S}) \quad \quad \text{and} \quad \quad P_{\text{top}}^{(q)}(\cF_G,0,\mathbb{P}_{\underline{p}}) < P_{\text{top}}^{(a)}(\cF_G,0,\mathbb{P}_{\underline{p}}).$$
So, in this case, we get
$\sup_{\nu\, \in\, \mathcal{M}_G} \, h_\nu(\mathbb{S},\mathbb{P}_{\underline{p}}) < h_{\text{top}}(\mathbb{S}).$
\end{remark}

\begin{example}\label{example1}
Let $g_1: \mathcal{S}^1 \to \mathcal{S}^1$ and $g_2: \mathcal{S}^1 \to \mathcal{S}^1$ be the unit circle expanding maps given by $g_1(z)=z^2$ and $g_2(z)=z^3$ and consider the free semigroup $G$ generated by $G_1=\{id, g_1,g_2\}$. Their topological entropies are $\log 2$ and $\log 3$, respectively. Let $\mathbb S$ be the corresponding semigroup action. According to \cite[Section \S8]{CRV}, we have $h_{\text{top}}(\cF_G)=\log 5 \sim 1.609$, $h_{\text{top}}(\mathbb S)=\log (\frac{5}{2}) \sim 0.916$ and $P_{\text{top}}^{(q)}(\cF_G,0,\mathbb{P}_{\underline{2}}) = \frac{\log 3 + \log 2}{2} \sim 0.896$.
\end{example}

\begin{remark}\label{sse.fibered metric entropy}
Each time we fix $\omega=\omega_1\omega_2\cdots \in \Sigma^+_p$, we restrict the semigroup action to a sequential dynamical system, we will denote by $\omega$-SDS, whose orbits are the sequences $(f_\omega^n(x))_{n \,\in\, {\mathbb N}_0; \, x \, \in \, X}$.
Given $\omega \in \Sigma^+_p$ and $\nu \in \mathcal{M}_G$, we may define the measure-theoretic entropy of the $\omega$-SDS by
$h_\nu(\text{$\omega$-SDS})=\sup_\beta \, h_\nu(\text{$\omega$-SDS}, \beta)$, where $\beta$ is any measurable finite partition of $X$,
\begin{equation}\label{eq.beta omega}
h_\nu(\text{$\omega$-SDS}, \beta) = \lim_{n\, \to\, + \infty}\, \frac{1}{n}\,H_\nu(\beta^n_0(\omega))
\end{equation}
and $\beta^n_0(\omega)$, $\beta^n_1(\omega)$ are as in \eqref{eq.beta}. Then, using the Dominated Convergence Theorem, it is not hard to prove that,
for every probability measure $\nu \in \mathcal{M}_G$, we have
$$h_\nu(\mathbb{S}, \mathbb{P}) \leq \int_{\Sigma^+_p}\, h_\nu(\text{$\omega$-SDS})\,d\,\mathbb{P}(\omega).$$
\end{remark}

\section{Proof of Theorem~\ref{Thm:D}}\label{sec:Recdelta}

We will start this section recalling the notion of orbital specification property introduced in \cite{RoVa1} and a few facts about recurrence by the skew product associated to a free semigroup action. The reader acquainted with this preliminary information may omit the next two subsections.

\subsection{Orbital specification}

We say that the continuous semigroup action $\mathbb{S}: G\times X \to X$ associated to the finitely generated semigroup $G$ satisfies the \emph{weak orbital specification property} if, for any $\vep>0$, there exists $q(\vep)>0$ such that, for any $q \ge q(\vep)$, we may find a set $\tilde G_{q}\subset G_{q}^*$ satisfying
$\lim_{q\to\infty}\,\sharp \,\tilde G_p /\sharp \,G_p^*=1$
and for which the following shadowing property holds: for any $h_{q_j}\in \tilde G_{q_j}$ with $q_j\ge q(\vep)$, any points $x_1, \dots, x_k \in X$, any natural numbers $n_1, \dots, n_k$ and any concatenations
$\underline g_{n_j, j}= g_{i_{n_j}, j} \dots g_{i_2,j} \, g_{i_1,j} \in G_{n_j}$ with $1\leq j \leq k$, there exists $x\in X$ satisfying
$\text{dist}(\underline g_{{\ell}, 1} (x) \; , \;\underline g_{{\ell}, 1} (x_1) ) < \vep, \forall \ell = 1, \dots, n_1$
and
$\text{dist }(\underline g_{{\ell}, j} \, \underline  h_{q_{j-1}} \, \dots \, \underline g_{{n_2}, 2} \, \underline h_{q_1} \, \underline g_{{n_1}, 1} (x) \; , \; \underline g_{{\ell}, j} (x_j)) <  \vep$ for all $j=2, \dots, k$ and $\ell = 1, \dots, n_j.$
If $\tilde G_{ p}$ can be taken equal to $G^*_{ p}$, we say that $\mathbb{S}$ satisfies the \emph{strong orbital specification property}. If the point $x$ can be chosen in $\text{Per}(G)$, then we refer to this property as the \emph{periodic orbital specification property}. For instance, it is true for finitely generated semigroups of topologically mixing Ruelle expanding maps (cf. \cite[Theorem~16]{RoVa1}).

\subsection{First return times}

Although the recurrence for a semigroup action $\mathbb S$ and for the random dynamical system modeled by the skew product $\cF_G$ are not the same, they are nevertheless bonded. Given a measurable subset $A$ of $X$ and any $x \in A$, we may define \emph{the first return of $x$ to $A$ by the semigroup action} as follows
\begin{equation}
n^\mathbb{S}_A(x)= \left\{\begin{array}{ll}
\inf\,\left\{n_A^\omega(x) \colon \omega \in \Sigma_p^+ \right\} & \mbox{if this set is nonempty}\\
+\infty & \mbox{otherwise.}
\end{array}
\right.
\end{equation}
Then
$n^\mathbb{S}_A(x) = \inf \,\{k\ge 1 \colon \cF_G^k \,(\Sigma_p^+ \times \{x\}) \cap (\Sigma_p^+ \times A) \neq\emptyset\}.$
Moreover, given $B \subset \Sigma_p^+ \times X$, we may take \emph{the shortest return time of $B$ to itself by the skew product $\cF_G$}, that is,
$$\mathcal{T}^{\cF_G}(B)=\inf\, \{k \in \mathbb N \colon \cF_G^k (B) \cap  B \neq \emptyset \}.$$
In particular, if $B=\Sigma_p^+\times A$, we obtain
$$\inf_{x\,\in\, A} n^\mathbb{S}_A(x) = \mathcal{T}^{\cF_G}(\Sigma_p^+\times A) = \inf_{\omega \,\in\, \Sigma_p^+} \mathcal{T}^\omega(A)$$
and (see Definition~\ref{de:action-return}) 
\begin{equation}\label{eq:minimalFG}
\mathcal{T}^{\mathbb{S}}(A)
	= \inf \,\{k\ge 1 \colon \cF_G^k (\Sigma_p^+ \times A) \cap (\Sigma_p^+ \times A) \neq\emptyset \}
	= \mathcal{T}^{\cF_G}(\Sigma_p^+ \times A).
\end{equation}
The pointwise return time functions for the semigroup action $\mathbb S$ and the skew product $\cF_G$ are also related: by \eqref{eq:recskew}, given a measurable set $A \subset X$, for every $x \in X$ and $\omega \in \Sigma_p^+$ we have
$$n_A^\omega(x) = n^{\cF_G}_{\Sigma_p^+ \times A}(\omega,x)= \text{ first return time of } (\omega,x) \text{ to the set } {\Sigma_p^+ \times A} \text{ by } \cF_G.$$

\subsection{Shortest returns of balls and Lyapunov exponents}

In the special case of semigroups of topologically mixing expanding maps, it is known that the skew product map $\cF_G$ satisfies the periodic specification property (see e.g. \cite[Theorem~28]{RoVa1}). Moreover, if $\mathbb{P} \times \nu$ has positive entropy with respect to $\cF_G$ then, using ~\eqref{eq:minimalFG}, for $\mathbb{P} \times \nu$-almost every $(\omega,x)$, one has (cf. \cite{ACS03,Var09})
\begin{equation*}
\lim_{\vep\,\to\,0} \limsup_{n\,\to\, \infty} \frac{\mathcal{T}^{\mathcal F_G}(B_\vep((\omega,x),n))}{n}
	=\lim_{\vep\,\to\,0} \liminf_{n\,\to\, \infty} \frac{\mathcal{T}^{\mathcal F_G}(B_\vep((\omega,x),n))}{n}
	=1
\end{equation*}
where $B^\omega_{\delta}(x,n)= \{y\in X \colon d(f_{ \omega}^j(x), f_{ \omega}^j(y))<\delta, \,\,\,\forall \,0\le j \le n-1\}$ stands for the dynamical ball with center $x$, radius $\delta$ and length $n$ for the dynamics $(f_\omega^n)_{n\ge 1}$. 
The next result generalizes this statement, employing a notion of metric entropy of the semigroup action whose definition will be given in Subsection~\ref{sse.metric entropy}.

\begin{proposition}\label{prop:tau} Let $G$ be the semigroup generated by $G_1=\{id, g_1, \dots, g_p\}$, where the elements in $G_1^*$ are $C^1$ expanding maps on a compact connected Riemannian manifold $X$, satisfy the orbital specification property and preserve a Borel probability measure $\nu$ on $X$. Consider a $\sigma$-invariant Borel probability measure $\mathbb P$ on $\Sigma_p^+$ such that 
$h_\nu(\mathbb S, \mathbb P) > 0$. Assume also that  
$\mathbb P \times \nu$ is ergodic with respect to $\cF_G$. Then, for $\mathbb{P}$-almost every $\omega$ and $\nu$-almost every $x\in X$, we have
$$\lim_{\delta\,\to\, 0} \limsup_{n\,\to\,+\infty} \,\frac{ \mathcal{T}^\omega (B^\omega_\delta(x,n)) }{n}
= \lim_{\delta\,\to\, 0} \liminf_{n\,\to\,+\infty} \,\frac{ \mathcal{T}^\omega (B^\omega_\delta(x,n)) }{n} = 1.$$
\end{proposition}

\begin{proof}
Firstly, observe that, as we are considering the product metric in $\Sigma_p^+ \times X$, then $B_\delta((\omega,x)) = B_\delta(\omega) \times B_\delta(x)$ for every $(\omega,x) \in \Sigma_p^+ \times X$ and any $\delta>0$. Moreover, dynamical balls with respect to the skew product dynamics $\mathcal{F}_G$ are in fact dynamical balls for the random composition of dynamics; that is, for every $n\ge 1$,
\begin{equation}\label{eq:inclusions1}
B_\delta((\omega,x), n) = \bigcup_{\theta \,\in\, B_\delta(\omega,n)} \,\{\theta\} \times B^{\theta}_{\delta}(x,n).
\end{equation}
Besides, if we take $\Lambda=\max_{1\,\le\, i\,\le\, p; \,x\,\in\, X} \|Dg_i(x)\|$ and $\lambda=\min_{1\,\le\, i\,\le\, p; \, x\,\in\, X} \| Dg_i(x)\|$, then clearly
\begin{equation}\label{eq:inclusions2}
B_\delta(\omega,n) \times B_{\delta \Lambda^{-n}}(x)
	\subset B_\delta((\omega,x), n)
	\subset B_\delta(\omega,n) \times B_{\delta \lambda^{-n}}(x)
\end{equation}
for every $x\in X$ and $n\ge 1$, which implies that the corresponding first return times are in decreasing order.

The periodic orbital specification property of the skew-product guarantees that, for any $\delta>0$, there exists $N_\delta\ge 1$ such that, given $n\ge 1$, we may find a periodic point $y \in B^\omega_\delta(x,n) \cap \text{Fix} \,(f_\omega^{n+N_\delta})$. In particular, $\mathcal{T}^\omega (B^\omega_\delta(x,n)) \le n + N_\delta$ and, consequently,
$$ \lim_{\delta\,\to\, 0}\limsup_{n\,\to\,+\infty} \frac{\mathcal{T}^\omega (B^\omega_\delta(x,n))}{n} \le 1.$$
To complete the proof we are left to show that, for $(\mathbb{P} \times \nu)$-almost every $(\omega,x)$,
\begin{equation}\label{eq:bound}
\lim_{\delta\,\to\, 0}\liminf_{n\,\to\,+\infty} \frac{\mathcal{T}^\omega (B^\omega_\delta(x,n))}{n} \ge 1.
\end{equation}
We will argue as in \cite[pages 2372--2373]{Var09}. 
Notice that, as $\mathbb{P} \times \nu$ is ergodic and $h_\nu(\mathbb{S}, \mathbb{P}) > 0$,
Theorem~2.1 of \cite{Zhu} informs that, for $\mathbb P \times \nu$-almost every $(\omega,x)$,
\begin{equation}\label{Zhu}
\lim_{\delta\,\to\, 0} \limsup_{n\,\to\,+\infty} - \frac1n \log \nu (B^\omega_\delta(x,n)) = \lim_{\delta\,\to\, 0} \liminf_{n\,\to\,+\infty} - \frac1n \log \nu (B^\omega_\delta(x,n)) > 0
\end{equation}
and that
$$h_\nu(\mathbb{S}, \mathbb{P}) = \int \big[ \lim_{\delta\, \to\, 0} \limsup_{n\,\to\,+\infty} - \frac1n \log \nu (B_\delta^\omega(x,n)) \big] \; d(\mathbb{P} \times \nu)(\omega,x).$$
Take now a finite measurable partition $\beta$ of $X$ satisfying $\nu(\partial \beta)=0$ and $h_\nu(\mathbb{S}, \mathbb{P}, \beta)>0$. Let $V_\delta(\partial \beta)$ stand for the neighborhood of size $\delta$ of $\Sigma_p^+ \times \partial \beta$ in $\Sigma_p^+ \times X$; notice that $(\mathbb{P} \times \nu) (V_\delta(\partial \beta))=\nu (V_\delta(\partial \beta))$. The Random Ergodic Theorem (cf. Subsection~\ref{thm:RET}) assures that, for any
small $\gamma>0$, there exists $\delta >0$ such that, at $\mathbb{P} \times \nu$-almost everywhere, one has
\begin{equation}\label{eq:boundaryintersection}
\frac1n \,\sum_{j=0}^{n-1} \delta_{V_\delta(\partial \beta)}(\cF_G^j(\omega,x)) \le  2\,(\mathbb{P} \times \nu) (V_\delta(\partial \beta)) <\gamma.
\end{equation}

Fix $\omega \in \Sigma_p^+$ in the $\mathbb P$-full measure subset of $\Sigma_p^+$ so that \eqref{Zhu} and ~\eqref{eq:boundaryintersection} hold.
 As the semigroup action $\mathbb S$ is ergodic (cf. definition in Subsection~\ref{sse:Kac}),
for any $\xi,\,\varepsilon >0$ small enough there exist $N \in \mathbb N$ and a measurable set $E^\omega_\xi \subset X$ satisfying $\nu(E^\omega_\xi)> 1-\xi$,
\begin{equation}\label{eq:rand1}
e^{- \,n \,(h_\nu(\mathbb{S}, \,\mathbb{P},\,\beta)\,+\,\xi)} \le \nu(\beta_0^n(\omega)(x))\le e^{- \,n\, (h_\nu(\mathbb{S},\, \mathbb{P},\,\beta)\,-\,\xi)}
\end{equation}
and
$\sum_{j=0}^{n-1} \delta_{\cF_G^j\,(\omega,x)} \le \gamma n $
for all $\,x \in E^\omega_\xi$ and $n\ge N$.
Besides, by equation ~\eqref{eq:rand1}, there exists $K_\omega>0$ such that
$$K_\omega^{-1} e^{- \,n \,(h_\nu(\mathbb{S},\, \mathbb{P},\,\beta)\,+\,\xi)} \le \nu(\beta_0^n(\omega)(x)) \le K_\omega e^{- \,n\, (h_\nu(\mathbb{S},\, \mathbb{P},\,\beta)\,-\,\xi)}$$
for every $n\ge 1$ and $x\in E_\xi^\omega$. As $\xi>0$ was chosen arbitrary, in order to prove \eqref{eq:bound} for $\nu$-almost every $x$, it is enough to show, using Borel-Cantelli Lemma, that
$\nu \big(\{x \in E^\omega_\xi : \mathcal{T}^\omega (B^\omega_\delta(x,n)) \leq (1-\xi) \,n \}\big)$
is summable for every small $\vep$.

We proceed covering the dynamical ball $B^\omega_\vep(x,n) \subset X$ by a collection $\tilde \beta^n_0(\omega)$ of partition elements in $\beta_0^n(\omega)$.
If $\delta>0$ is chosen small enough, then \eqref{eq:boundaryintersection} implies that the piece of orbit $(f^j_\omega(x))_{j=0}^n$ enters the $\delta$-neighborhood of $\partial\beta$ in at most $\gamma \,n$ iterates.
The argument used in  \cite[Lemma~3.2]{Var09} implies that, for any $\alpha>0$, there exist $\gamma>0$ and $\delta>0$ (given by \eqref{eq:boundaryintersection}) so that $B^\omega_\vep(x,n) \subset X$ is covered by
a collection  $\tilde \beta^n_0(\omega)$ of at most $e^{\alpha n}$ partition elements of $\beta_0^n(\omega)$, for every $x\in E_\xi^\omega$.
Therefore,
\begin{align*}
\nu \big(\{ x \in E^\omega_\xi : \mathcal{T}^\omega (B^\omega_\delta(x,n)) \leq (1-\xi) \,n \}\big)
    & = \sum_{k=0}^{(1-\xi) n} \nu \big(\{ x \in E^\omega_\xi : \mathcal{T}^\omega (B^\omega_\delta(x,n)) = k \}\big) \\
    & \le \sum_{k=0}^{(1-\xi) n}  \sum_{\substack{Q \,\in\, \tilde \beta_0^n(\omega) \\  f_\omega^k(Q) \,\in\, \tilde \beta_0^n(\omega)}}  \nu(E_\xi^\omega \cap Q).
\end{align*}
Note that $B_\delta^\omega(x,n)$ is covered by at most $e^{\alpha n}$ elements of $\tilde \beta_0^n(\omega)$ and,
among these, every $Q \in \tilde \beta_0^n(\omega)$ satisfying $f_\omega^k(Q) \in \tilde \beta_0^n(\omega)$ is determined by the first $k$ elements of the partition $\beta$ that are visited under the iterations of $f^j_\omega$, $1\le j \le k$. Thus 
\begin{align*}
\nu \big(\{ x \in E^\omega_\xi \colon & \mathcal{T}^\omega (B^\omega_\delta(x,n)) \leq (1-\xi) \,n \}\big) \\
    & \le   K_\omega e^{-\, n\, (h_\nu(\mathbb{S}, \mathbb{P},\beta)\,-\,\xi)}\,  e^{\alpha n}
 	   \sum_{k=0}^{(1-\xi) n} \# \{ Q \,\in\, \tilde \beta_0^n(\omega) :  f_\omega^k(Q) \,\in\, \tilde \beta_0^n(\omega) \}\\
& \le   K_\omega e^{-\, n\, (h_\nu(\mathbb{S}, \mathbb{P},\beta)\,-\,\xi)}\,  e^{\alpha n}
 	   \sum_{k=0}^{(1-\xi) n} K_\omega e^{k \, (h_\nu(\mathbb{S}, \mathbb{P},\beta)\,+\,\xi)} \\
    & \le K_\omega^2 (1-\xi) n \, e^{-\, n\, (h_\nu(\mathbb{S}, \mathbb{P},\beta)\,-\,\xi)} \, e^{\alpha n} \,
 	   	  e^{(1-\xi)\, n\, (h_\nu(\mathbb{S}, \mathbb{P},\beta)\,+\,\xi)}
\end{align*}
and so it is summable provided that $\alpha,\xi$ are small enough. 
\end{proof}

If we restrict to either $C^1$ expanding or conformal maps on a Riemannian manifold, we obtain the following corollary and finish the proof of Theorem~\ref{Thm:D}.

\begin{corollary}\label{cor:applicationsskew}
Let $G$ be the semigroup generated by $G_1=\{id, g_1, \dots, g_p\}$, where the elements in $G_1^*$ are $C^1$ expanding maps on a compact connected Riemannian manifold $X$ preserving a common Borel probability measure $\nu$ on $X$, and $\mathbb P$ be a $\sigma$-invariant Borel probability measure on $\Sigma_p^+$. Assume that 
$\mathbb P \times \nu$ is ergodic with respect to $\cF_G$ and that $h_{\nu}(\mathbb S,\mathbb P )>0$.
Then, for $\mathbb{P}$-almost every $\omega$ and $\nu$-almost every $x\in X$, we have
\begin{equation}\label{eq:mc1a}
\frac{1}{\log \Lambda} \le \liminf_{\delta\,\to\, 0} \,\frac{\mathcal{T}^\omega (B_\delta(x))}{-\log \vep} \le \limsup_{\delta\,\to\, 0} \,\frac{\mathcal{T}^\omega (B_\delta(x)) }{-\log \vep} \le \frac{1}{\log \lambda}
\end{equation}
where $\Lambda \ge \lambda > 1$ are, respectively,
$\Lambda = \max_{\substack{1\,\le\, i\,\le\, p}}\, \|Dg_i\|_\infty$ and $\lambda= \min_{\substack{1\,\le\, i\,\le\, p}}\, \|Dg_i\|_\infty.$
\end{corollary}

\begin{proof}
As a consequence of \cite[Theorem~16]{RoVa1}, we know that every finitely generated free semigroup action by $C^1$ expanding maps on compact connected manifolds satisfies the orbital specification property.
Moreover, for every $x\in X$, $n \in \mathbb N$, $\omega \in \Sigma_p^+$ and $\delta>0$, we have $B_{\Lambda^{-n}\delta}(x) \subset B^\omega_{\delta}(x,n) \subset B_{\lambda^{-n}\delta}(x)$.
Thus,if $\nu$ is a probability measure invariant by all elements in $G$, $\mathbb P \times \nu$ is ergodic
for $\cF_G$ and $h_{\nu}(\mathbb S,\mathbb P )>0$, we conclude from Proposition~\ref{prop:tau} that
$$\limsup_{\vep\,\to\, 0} \frac{\mathcal{T}^\omega(B_\delta(x))}{-\log \delta} \le  \lim_{\vep\,\to\, 0} \limsup_{n\,\to\,+\infty } \frac{\mathcal{T}^\omega(B_{\lambda^{-n}\vep}(x))}{-\log (\lambda^{-n} \vep)} \le  \lim_{\vep\,\to\, 0} \limsup_{n\, \to\, +\infty}  \frac{\mathcal{T}^\omega( B^\omega_{\vep}(x,n))}{n \log \lambda} = \frac1{\log\lambda}.
$$
The lower bound estimate is obtained by a similar reasoning.
\end{proof}

It is a straightforward outcome of \eqref{eq:mc1a} that,  
in the case of semigroups of $C^1$-expanding maps, one has
$$\limsup_{\delta \,\to\, 0} \,\frac{\mathcal{T}^\mathbb{S}(B_\vep(x)) }{-\log \vep} \le \frac{1}{\log \lambda}.$$
Although we believe that a similar lower bound holds, we have not been able to obtain it.

\begin{remark} Notice that, if the generators are not expanding maps one expects larger return times.
For instance, if the semigroup action is generated by circle rotations with rotation numbers in the interval $0 <\alpha_0\le  \alpha
\le\alpha_1 < 1$ then it is not hard to check that $\mathcal{T}^\mathbb{S}(B_\vep(x)) \le (\frac1{\alpha_1} + 1)\frac1\delta$ for every $\delta>0$ and $x\in \mathcal S^1$. However, it is not clear whether this bound is optimal.

\end{remark}

If, besides being expanding, all elements in $G^*_1$ are conformal, then $Dg_i(x)= \|Dg_i(x) \| Id$ and
$\det |Dg_i(x)| = \|Dg_i(x)\|^{\dim X}$ for every $x\in X$ and any $i \in \{1,2,\cdots,p\}$. In particular, it follows from Oseledets' Theorem that all the Lyapunov exponents of the skew product $\cF_G$ along the fiber direction are equal and coincide with
\begin{align*}
\chi: & = \frac1{\dim X} \lim_{n\,\to\,+\infty} \frac1n \log |\det Df_\omega^n (x)|
	= \frac1{\dim X} \lim_{n\,\to\, +\infty} \frac1n \sum_{j=0}^{n-1} \log \Big|\det \frac{\partial \cF_G}{\partial x} (\cF^j_G(\omega,x)\Big| \\
	& = \frac1{\dim X}  \int\int \log \Big|\det \frac{\partial \cF_G}{\partial x}(\omega,x)\Big| \, d\nu(x) \, d\mathbb{P}_{\underline a}(\omega) = \frac1{\dim X} \sum_{i=1}^p \; a_i \int \log |\det Dg_i | \, d\nu
\end{align*}
(notice that, in the last but one estimate, we have used the ergodicity of $\mathbb{P} \times \nu$). Moreover, the expanding nature of the generators in $G_1^*$ and the assumption that $\mathbb{P}=\mathbb{P}_{\underline a}$, for some probability vector $\underline a$, imply that $\chi >0$.
Observe also that as $\|Dg_i(x)\|=\det |Dg_i(x)|^\frac1{{\dim X}}$ then, as a consequence of the mean value theorem,
given $\vep>0$, for $\mathbb{P}_{\underline a} \times \nu$-almost every $(\omega,x)$ there exists $N=N(\omega,x)\ge 1$ such that, for all $n\ge N$,
$$B_{e^{-(\chi+\varepsilon)\,n}\,\delta}(x)	\subset B^\omega_{\delta}(x,n)	\subset B_{e^{-(\chi-\varepsilon)\,n}\,\delta}(x).$$
Then, an argument identical to the one used in the first part of this proof yields that
\begin{align*}
\frac1{\chi+\vep}
	\le \limsup_{\vep\,\to\, 0} \frac{\mathcal{T}^\omega(B_\delta(x))}{-\log \delta}
	\le  \frac1{\chi-\vep}.
\end{align*}
To obtain \eqref{eq:mc2a} it is enough to let $\vep$ go to $0$.
This completes the proof of Theorem~\ref{Thm:D}.

\begin{remark}
If the measure $\mathbb{P} \times \nu$ has positive entropy and $\cF_G$ has the specification property, then it is
a consequence of \cite{BSTV03, Var09} that, for $\mathbb{P}\times \nu$-almost every $(\omega,x)$,
$$\lim_{\delta\,\to\, 0} \limsup_{n\,\to\,+\infty} \frac{\mathcal{T}(B_\delta((\omega,x), n)) }{n}
	= \lim_{\delta\,\to\, 0} \liminf_{n\,\to\,+\infty} \frac{\mathcal{T}(B_\delta((\omega,x), n)) }{n}
	=1.$$
This differs substantially from the fibred assertion provided by Theorem~\ref{Thm:D}.
\end{remark}

\begin{example}
Consider the generators $g_1$ and $g_2$ of Example~\ref{example1}, the Lebesgue measure $Leb$ on $\mathcal{S}^1$ (which is invariant by both dynamics) and the symmetric random walk corresponding to the Borel probability measure $\mathbb{P}_{\underline{2}}$. The maps $g_1$ and $g_2$ are conformal, $C^1$ expanding, $|\det Dg_1(\cdot)|= 2$ and $|\det Dg_2(\cdot)|=3$. Besides, $\mathbb{P}_{\underline{2}}\times Leb$ is ergodic with respect to the skew product $\cF_G$ (it is even weak Bernoulli; cf. \cite{M85}). Therefore, by Corollary~\ref{cor:applicationsskew} and
Example~\ref{example1}, for $\mathbb P_{\underline 2}$-almost every $\omega$ in $\Sigma_2^+$ and $Leb$-almost every $z \in \mathcal{S}^1$, we have
\begin{equation*}
\lim_{\delta\,\to\, 0} \,\frac{\mathcal{T}^\omega (B(z,\delta)) }{-\log \vep} = \frac{2}{\log 3 + \log 2} > \frac{1}{h_{\text{top}}(\mathbb{S})}.
\end{equation*}
\end{example}

\section{Proof of Theorem~\ref{Thm:E}}\label{se:RQR}

In this section we examine the connection between maximum random hitting frequency and the size of sets when evaluated by different measures. 
We start reviewing these concepts. 
Denote by $\mathcal{M}_{\mathcal{F}_G}$ the set of all probability measures invariant by the skew product $\mathcal{F}_G$. For every $\mu\in \mathcal{M}_{\mathcal{F}_G}$ the marginal of $\mu$ on $X$ is the probability measure
$(\pi_X)_*(\mu):= \mu\circ \pi_X^{-1}$ where $\pi_X: \Sigma_p^+ \times X \to X$ is the natural projection.

\subsection{Random mean sojurns} Let $G$ be a finitely generated free semigroup, with corresponding action $\mathbb{S}$, and $\mathbb{P}$ a $\sigma$-invariant probability measure on $\Sigma_p^+$.

\begin{definition}\label{def:hit} For a measurable subset $A \subset X$, the \emph{maximum random hitting frequency} of $A$ with respect to $\mathbb{P}$ is given by
\begin{align}\label{eq:ge}
\gamma_\mathbb{P}(A)=
			\mathbb{P}-\text{esssup}\,\sup_{x\,\in\, X} \,\,\gamma_{\omega,x}(A)
\end{align}
where
\begin{align}\label{eq:gex}
\gamma_{\omega, x}(A) = \limsup_{n\,\to\,+\infty}\,\frac{\sharp\,\{0\leq i\leq n-1:f_\omega^i(x)\in A\}}{n}.
\end{align}
The \emph{absolute maximum hitting frequency} of $A$ with respect to $\mathbb{P}$ is defined by
\begin{align*}
\gamma(A)=\sup_{(\omega,x) \,\in\, \Sigma_p^+ \times X}\, \gamma_{\omega, x}(A).
\end{align*}
\end{definition}

Given a measurable set $A \subset X$, consider the upper bound of its sizes when estimated by all probability measures in $\mathcal{M}_{\mathcal{F}_G}$ which project on $\mathbb{P}$, that is,
\begin{align*}
\alpha_\mathbb{P}(A) = \sup_{\{\mu\, \in \, \mathcal{M}_{\mathcal{F}_G} \colon \pi_*\mu=\mathbb{P}\}}\,\mu(\Sigma_p^+\times A).
\end{align*}

\begin{lemma}\label{lemma:usc} If $A$ is a closed subset of $X$, then there exists an ergodic probability measure $\mu_A\in\mathcal M_{\mathcal{F}_G}$ with $\pi_* \mu_A=\mathbb{P}$ and such that
$\alpha_\mathbb{P}(A)=\mu_A(\Sigma_p^+ \times A).$ Moreover, the set of maximizing measures is compact.
\end{lemma}

\begin{proof}
Firstly, endow the space $\mathcal{M}_{\mathcal{F}_G}$ with the weak$^*$ topology. Therefore, as $A$ is closed, the functional
$$\Psi_A:\mu\in\mathcal M_{\mathcal{F}_G} \mapsto \mu(\Sigma_p^+\times A)$$
is upper semi-continuous (cf. \cite{Wa}). Moreover, $\mathcal M_{\mathcal F_G,\mathbb{P}} := \mathcal M_{\mathcal F_G} \cap \pi_*^{-1}(\mathbb{P})$ is a non-empty compact subset of $\mathcal{M}_{\mathcal{F}_G}$.
Hence, $\Psi_A$ attains a maximum in $\mathcal M_{\mathcal F_G,\mathbb{P}}$.

Let $\text{Erg}(\mathcal F_G,\mathbb{P})$ be the ergodic members of $\mathcal M_{\mathcal F_G,\mathbb{P}}$, and consider a measure $\xi_A \in \mathcal M_{\mathcal F_G,\mathbb{P}}$ where $\Psi_A$ attains its maximum, whose ergodic decomposition (cf. \cite{Wa, Ph}) in $\mathcal M_{\mathcal F_G,\mathbb{P}}$  is $\xi_A=\int_{\text{Erg}(\mathcal F_G,\mathbb{P})}\,\,m\,d\tau(m).$ As $\xi_A$ maximizes $\Psi_A$, we know that $m(\Sigma_p^+\times A) \leq \xi_A(\Sigma_p^+\times A)$ for every $m \in \mathcal M_{\mathcal F_G,\mathbb{P}}$. Therefore, as $\xi_A(\Sigma_p^+\times A)=\int_{\text{Erg}(\mathcal F_G,\mathbb{P})}\,\,m(\Sigma_p^+\times A)\,d\tau(m)$, we must have
$m(\Sigma_p^+\times A)=\xi_A(\Sigma_p^+\times A) = \alpha_\mathbb{P}(A)$
for $\tau$-almost every $m$. Thus, we may take an ergodic maximizing measure of $\Psi_A$, as claimed.

We observe that the upper semi-continuity of $\Psi_A$ also implies that the set of maximizing probability measures is compact. Indeed, if a sequence of measures $\xi_{A,n} \in \mathcal M_{\mathcal F_G,\mathbb{P}}$ satisfies $\lim_{n \,\to\, +\infty}\, \xi_{A,n}=\xi$ in the weak$^*$ topology and $\alpha_\mathbb{P}(A)=\xi_{A,n}(\Sigma_p^+ \times A)$ for every $n \in \mathbb{N}$, then $\xi \in \mathcal M_{\mathcal F_G,\mathbb{P}}$ and, as $\Psi_A$ is upper semi-continuous, we conclude that
\begin{eqnarray*}
\alpha_\mathbb{P}(A) &=& \lim_{n\, \to\, +\infty} \xi_{A,n}(\Sigma_p^+ \times A) = \lim_{n\, \to\, +\infty} \Psi_A(\xi_{A,n}) \leq \Psi_A(\xi) = \xi(\Sigma_p^+ \times A) \\
&\leq&  \sup_{\{\mu\,\in\, \mathcal{M}_{\mathcal{F}_G} \colon \pi_*\mu=\mathbb{P}\}}\,\mu(\Sigma_p^+\times A) = \alpha_\mathbb{P}(A)
\end{eqnarray*}
so $\Psi_A(\xi)=\alpha_\mathbb{P}(A).$
\end{proof}

We are now ready to compare the rates of visits with the size of the visited set.

\begin{proposition}\label{prop:hitting}
Let $\mathbb{P}$ be a $\sigma$-invariant probability measure on $\Sigma_p^+$. Then:
\begin{enumerate}
\item $\alpha_\mathbb{P}(A) \le \gamma_\mathbb{P} (A) \le \gamma(A)$ for every measurable set $A\subset X$.
\item If $\mathbb{P}$ is ergodic, one has:
\begin{enumerate}
\item For every $x\in X$, there exists an $\cF_G$-invariant probability measure $\mu_{\mathbb{P},x}$ such that $\pi_*(\mu_{\mathbb{P},x})=\mathbb{P}$ and, for every closed set $A\subset X$,
	$\gamma_{\omega,x}(A) \le \mu_{\mathbb{P},x}(\Sigma_p^+ \times A).$
\item $\gamma_{\mathbb{P}}(A) = \alpha_\mathbb{P}(A)$ for every closed set $A \subset X$.
\end{enumerate}
\end{enumerate}
\end{proposition}

\begin{proof} Consider a measurable set $A \subset X$. The inequality $\gamma_\mathbb{P} (A)\le \gamma(A)$ is immediate. Conversely, if $\mu$ is an $\cF_G$-invariant and ergodic probability measure on $\Sigma_p^+ \times X$, then it follows from Birkhoff's Ergodic Theorem that
$$
\lim_{n\,\to\,+\infty}\frac{\sharp\,\{0\leq i\leq n-1: f_\omega^i(x)\in A\}}{n} = \lim_{n\,\to\,+\infty}\frac{\sharp\,\{0\leq i\leq n-1: \cF_G^i(\omega,x)\in \Sigma_p^+ \times A\}}{n} = \mu(\Sigma_p^+\times A)
$$
for $\mu$-almost every $(\omega,x)$. Thus, taking the supremum and the essential supremum in the first term of the previous equalities, we conclude that
$$
\gamma_\mathbb{P}(A)= \mathbb{P}-\text{esssup}\,\sup_{x\,\in\, X}\, \limsup_{n\,\to\,+\infty}\,\frac{\sharp\,\{0\leq i\leq n-1: f_\omega^i(x)\in A\}}{n} \ge \mu(\Sigma_p^+\times A).
$$
This proves (1) since $\mu$ is an arbitrary ergodic measure and, as a consequence of the Ergodic Decomposition Theorem, we have 
$$
\sup_{\{\mu\,\in\, \mathcal{M}_{\mathcal{F}_G} \colon \pi_*\mu=\mathbb{P}\}}\,\mu(\Sigma_p^+\times A)
	= \sup_{\{\mu\,\in\, \mathcal{M}_{\mathcal{F}_G} \colon \pi_*\mu=\mathbb{P} \; \& \; \text{$\mu$ is ergodic} \}}\,\mu(\Sigma_p^+\times A).
$$

We proceed to prove (2). Assume that $\mathbb{P}$ is ergodic and take a point $\omega\in \Sigma_p^+$ in the ergodic basin of the measure $\mathbb{P}$
$$\mathcal{B}(\mathbb{P})= \{\omega \in \Sigma^+_p \colon \lim_{n\,\to\,+\infty}\,\frac{1}{n}\,\sum_{j=0}^{n-1}\, \delta_{\sigma^j(\omega)} = \mathbb{P}\}.$$
Fix $x\in X$ and a closed set $A\subset X$. From Definition~\ref{def:hit}~\eqref{eq:gex}, we may find a subsequence $(n_k=n_k(x))_{k \,\in\, \mathbb{N}}$ going to $+\infty$ and such that
$\gamma_{\omega,x}(A) = \lim_{k\,\to\,+\infty}\,\frac{1}{n_k}\,\sharp\,\{0\leq i\leq n_k-1 : f_\omega^i(x)\in A\}.$
By compactness of the set of probability measures on the Borel subsets of $\Sigma_p^+ \times X$, the sequence of measures $(\mu_k)_{k\,\ge\, 1}$ given by
$ \mu_k:=\frac1{n_k} \,\sum_{i=0}^{n_k-1} \,\delta_{\cF_G^i(\omega,x)}$
admits a weak$^*$ convergent subsequence to some $\cF_G$-invariant probability measure $\mu_{\mathbb{P},x}$. Assume, without loss of generality, that $\mu_{\mathbb{P},x}=\lim_{k\,\to\,+\infty} \,\mu_k$. By the continuity of the projection $\pi_*$ and the choice of $\omega$, one has
$$
 \pi_*(\mu_{\mathbb{P},x})= \lim_{k\,\to\,+\infty} \, \pi_* \big(\frac1{n_k} \, \sum_{i=0}^{n_k-1} \, \delta_{\cF_G^i(\omega,x)} \big)
 	= \lim_{k\,\to\,+\infty}\,  \frac1{n_k} \,\sum_{i=0}^{n_k-1} \, \delta_{\sigma^i(\omega)}
	= \mathbb{P}.
 $$
Moreover, as $\Sigma_p^+\times A$ is a closed set and $(\mu_k)_{k \,\in\, \mathbb{N}}$
is weak$^*$ convergent to $\mu_{\mathbb{P},x}$, we get
$$
\gamma_{\omega,x}(A) = \lim_{k\,\to\,+\infty}\,\frac{\, \sharp\, \{0\leq i\leq n_k-1:f_\omega^i(x)\in A\}}{n_k}
	= \lim_{k\,\to\,+\infty}\, \mu_k(\Sigma_p^+\times A) \le \mu_{\mathbb{P},x}(\Sigma_p^+\times A).
$$
This ends the proof of the item (2(a)) and also implies that
$$\gamma_{\mathbb{P}}(A)= \mathbb{P}-\text{esssup}\,\sup_{x\,\in\, X} \,\,\gamma_{\omega,x}(A) \le \mu_{\mathbb{P},x}(\Sigma_p^+\times A)$$
for every closed set $A \subset X$. Finally, notice that items (1) and (2(a)) together yield the equality $\gamma_{\mathbb{P}}(A) = \alpha_\mathbb{P}(A)$ for every closed set $A \subset X$.
\end{proof}

To complete the proof of Theorem~\ref{Thm:E} we have just to assemble the statements of Lemma~\ref{lemma:usc} and (2(b)) of Proposition~\ref{prop:hitting}, and take the marginal $(\pi_X)_*(\mu_A)$.


\medskip

\begin{example}
Let $g_1: [0,1] \to [0,1]$ and $g_2: [0,1] \to [0,1]$ be the maps given by $g_1(x)=4x\,(1-x)$ and $g_2(x)= 2x\mod 1$ and consider the continuous semigroup $G$ generated by $G_1=\{Id, g_1, g_2\}$. Given $n \in \mathbb{N}$, take $\ell_1 \in \left[\sin\,\frac{\pi}{2(2^n+1)}, \sin\,\frac{\pi}{2(2^{n-1}+1)}\right)$
and $\ell_2 \in \left[\frac{1}{2^n+1},\frac{1}{2^{n-1}+1}\right)$.
By Theorems 3 and 5 of \cite{Je}, if $A_{\ell_i}=[\frac{1-\ell_i}{2},\frac{1+\ell_i}{2}]$, then the maximum hitting frequency for the map $g_i$ is equal to $\gamma_i(A_{\ell_i})=\frac{1}{n}$ ($i=1,2$). Moreover, there are periodic points $z_1, z_2\in [0,1]$ 
with period $n$ and $2n$ for $g_1$ and $g_2$, respectively, whose orbits hit $A_{\ell_i}$ with maximum frequency ($i=1,2$).

Let $\mathbb P=\frac{1}{3} \delta_{\bar 1}+\frac{2}{3} \delta_{\bar 2}$ and $\mu=\frac{1}{3n}\left(\sum_{j=1}^{n}\delta_{\mathcal{F}_G^j(\overline{1},z_\alpha)}+\sum_{j=1}^{2n}\delta_{\mathcal{F}_G^j(\overline{2},z_\beta)}\right)$. Then $\mu$ is $\mathcal{F}_G$-invariant and, although $\mathbb P$ is not ergodic, we have $\pi_\ast \mu=\mathbb P$. Besides, by \cite{Je}, for all $\ell \in \left[\frac{1}{2^n+1},\frac{1}{2^{n-1}+1}\right)\,\cap \,[\sin\,\frac{\pi}{2(2^n+1)}, \sin\,\frac{\pi}{2(2^{n-1}+1)})$, we obtain $ \gamma_{\mathbb P}(A_\ell)=\frac{1}{n}=\alpha_{\mathbb P}(A_\ell)$.
\end{example}

\medskip

\subsection*{Acknowledgements}
MC has been financially supported by CMUP (UID/MAT/00144/2013), which is funded by FCT (Portugal) with national (MEC) and European structural funds (FEDER) under the partnership agreement PT2020. FR and PV were partially supported by BREUDS. PV has also benefited from a fellowship awarded by CNPq-Brazil and is grateful to the Faculty of Sciences of the University of Porto for the excellent research conditions. The authors are grateful to Jerome Rousseau for his valuable comments.

\end{document}